\documentclass{amsart}

\usepackage{tgtermes}
\usepackage{mathptmx} 
\usepackage[scaled=.92]{helvet}
\usepackage{amsthm,amsmath,amssymb}
\usepackage{mathrsfs}
\usepackage[numbers]{natbib} 
\usepackage[colorlinks,citecolor=blue,urlcolor=blue]{hyperref}
\usepackage{enumerate}

\newtheorem{theorem}{Theorem}[section]
\newtheorem{lemma}[theorem]{Lemma}

\newtheorem{claim}[theorem]{Claim}

\theoremstyle{definition}
\newtheorem{definition}[theorem]{Definition}

\newtheorem{example}[theorem]{Example}
\theoremstyle{remark}

\newcommand{\bpi}{\mathnormal{\mathsf{BPI}}}
\newcommand{\zfc}{\mathnormal{\mathsf{ZFC}}}
\newcommand{\zf}{\mathnormal{\mathsf{ZF}}}
\newcommand{\ac}{\mathnormal{\mathsf{AC}}}
\newcommand{\ma}{\mathnormal{\mathsf{MA}}}
\newcommand{\fa}{\mathnormal{\mathsf{FA}}}
\newcommand{\dc}{\mathnormal{\mathsf{DC}}}
\newcommand{\lo}{\mathnormal{\mathrm{LO}}}
\newcommand{\wo}{\mathnormal{\mathrm{WO}}}

\DeclareMathOperator{\dom}{dom}

\author[D. Fern\'andez]{David Fern\'andez-Bret\'on}
\address{Department of Mathematics\\
University of Michigan\\
2074 East Hall, 530 Church Street \\
Ann Arbor, MI 48109-1043, U.S.A.}
\email{djfernan@umich.edu}
\urladdr{http://www-personal.umich.edu/\textasciitilde djfernan/}
 
\author[E. Lauri]{Elizabeth Lauri}
\address{Department of Mathematics\\
University of Connecticut\\
341 Mansfield Road U1009\\
Storrs, Connecticut 06269-1009}
\email{elizabeth.lauri@uconn.edu}


\begin{document}

\title[A characterization of the BPI]{A characterization of the Boolean Prime Ideal theorem in terms of forcing notions}

\begin{abstract}
  For certain weak versions of the Axiom of Choice (most notably, 
  the Boolean Prime Ideal theorem), we obtain equivalent formulations 
  in terms of partial orders, and filter-like objects within them 
  intersecting certain dense sets or antichains. This allows us to 
  prove some consequences of the Boolean Prime Ideal theorem using 
  arguments in the style of those that use Zorn's Lemma, or Martin's 
  Axiom.
\end{abstract}

\maketitle

\section{Introduction}

It is well-known that the Axiom of Choice is equivalent to Zorn's Lemma, a 
statement about the existence of certain elements in some partial orders, 
which allows us to prove existence results by defining a partial order of 
approximations to the object whose existence is being established. Very similar 
in spirit are the so-called forcing axioms, which are combinatorial principles 
(that typically go beyond the $\zfc$ axioms, while consistent with these, so 
that they can be used for consistency proofs) that also involve the use 
of partial orders in their application. In order to state precisely what 
a forcing axiom is, we will proceed to introduce the necessary 
definitions.

\begin{definition}
Let $\mathbb P$ be a partially ordered set (and denote the corresponding 
partial order by $\leq$). Then
\begin{enumerate}
\item We will typically refer to elements of $\mathbb P$ as \textbf{conditions},
\item we will say that the condition $p$ \textbf{extends} the condition $q$ if 
$p\leq q$,
\item we will say that the two conditions $p,q$ are \textbf{compatible}, 
which we will denote by $p\not\perp q$, if 
they have a common extension, i.e. if there exists a condition $r$ such 
that $r\leq p$ and $r\leq q$,
\item we will say that the two conditions $p,q$ are \textbf{incompatible}, 
denoted $p\perp q$, if they are not compatible,
\item a subset $A\subseteq\mathbb P$ will be called an \textbf{antichain} if 
any two distinct conditions $p,q\in A$ must be incompatible,
\item we say that a subset $D\subseteq\mathbb P$ is \textbf{dense} 
if every condition has an extension in $D$, i.e. 
$(\forall p\in\mathbb P)(\exists q\in D)(q\leq p)$,
\item a subset $G\subseteq\mathbb P$ will be called a \textbf{filter} if 
it is closed upwards (this is, if
$(\forall p\in G)(\forall q\in\mathbb P)(p\leq q\Rightarrow q\in G)$), and 
for every $p,q\in G$ there exists an $r\in G$ which extends both $p$ and $q$,
\item if $\mathcal D$ is a family of dense subsets of $\mathbb P$ (respectively, 
if $\mathcal A$ is a family of antichains) then the filter $G$ will be called 
\textbf{$\mathcal D$-generic} if it intersects every element of $\mathcal D$ 
(respectively, \textbf{$\mathcal A$-generic} if it intersects every element 
of $\mathcal A$).
\end{enumerate}
\end{definition}

The previous definition contains all the terminology needed to talk about 
forcing axioms. A partial order $\mathbb P$ is said to be \textbf{c.c.c.} if 
every antichain is countable. Historically, the first example of a forcing 
axiom would be the one that is known as Martin's Axiom\footnote{For an 
introduction to Martin's Axiom and its consequences, 
see~\cite[Chapter II.2]{kunen}.}, abbreviated $\ma$.

\begin{tabbing}
  $\ma$ \quad \= For every c.c.c. partial order $\mathbb P$ and every 
family $\mathcal D$ of dense sets \\
 \> with $|\mathcal D|<\mathfrak c$, there exists a $\mathcal D$-generic filter.
\end{tabbing}

Other variations of this combinatorial principle have been proposed 
over the years, and this has eventually led to the formulation of an 
abstract template for a forcing axiom, which is as follows: let 
$\mathcal X$ be a class of partial orders, and $\kappa$ a cardinal. 
Then we can state the forcing axiom

\begin{tabbing}
  $\fa_\kappa(\mathcal X)$ \quad \= For every partial order $\mathbb P\in\mathcal X$ and every 
family $\mathcal D$ of dense sets \\
 \> satisfying $|\mathcal D|\leq\kappa$, there exists a $\mathcal D$-generic filter.
\end{tabbing}

Not all forcing axioms necessarily fall neatly into this template, 
so sometimes we will state some combinatorial principles that deviate 
slightly from it. For instance, we might want to use 
something other than a cardinal number in the place of the subindex 
in our template. An obvious example would be $\fa_\infty(\mathcal X)$, 
which should be interpreted as the statement that for every partial order 
$\mathbb P\in\mathcal X$ and every family $\mathcal D$ of dense sets (where 
$\mathcal D$ can be completely arbitrary, without restrictions of cardinality 
or of any other property), there exists a $\mathcal D$-generic filter. 
Similarly, expressions such as $\fa_{­<\kappa}(\mathcal X)$ should be given 
the obvious meaning.

This notation allows us to state several forcing axioms in a compact 
way; for example $\ma$ becomes the statement $\fa_{<\mathfrak c}(\text{c.c.c.})$, 
the combinatorial principle $\mathfrak t=\mathfrak c$ is equivalent 
(by Bell's Theorem~\cite{bell}, together with Malliaris and 
Shelah's~\cite{maryanthe1,maryanthe2} recent result that 
$\mathfrak p=\mathfrak t$) to 
$\fa_{<\mathfrak c}(\sigma\text{-centred})$\footnote{Recall that a 
partial order is $\sigma$-centred if it can be written as the union 
of countably many filters.}, and 
the combinatorial principle $\mathop{\mathrm{cov}}(\mathcal M)=\mathfrak c$ is 
equivalent to $\fa_{<\mathfrak c}(\text{countable})$. The Proper Forcing 
Axiom is, of course, the statement $\fa_{\omega_1}(\text{proper})$\footnote{We 
omit the definition of a proper partial order, since it will not be relevant 
for this paper, but direct the interested reader to~\cite{shelah} 
or~\cite{abraham}.}.

Although Zorn's Lemma is a statement that concerns partial orders, at 
first sight it does not look like a forcing axiom, for the object whose 
existence it asserts is not a filter. One of the first results in 
this paper shows that this first impression is misguided, and it is 
possible to rephrase Zorn's Lemma as a perfectly legitimate forcing 
axiom. This continues a certain 
line of research, that has been pursued~\cite{shannon}
in the past, concerning the possibility of expressing 
the Axiom of Choice, or some weak versions of it, as forcing axioms. 
The most immediate example of this is the following 
result\footnote{This result is attributed to 
Todor\v{c}evi\'c, although Goldblatt~\cite{goldblatt} came quite 
close to stating it.}. Recall that the principle of Dependent Choice, 
abbreviated $\dc$,  
is the statement that, for every set $X$ equipped with a binary relation 
$R$ satisfying $(\forall x\in X)(\exists y\in X)(x\ R\ y)$, there 
exists a sequence $\langle x_n\big|n<\omega\rangle$ such that 
$(\forall n<\omega)(x_n\ R\ x_{n+1})$.

\begin{theorem}[Todor\v{c}evi\'c]\label{stevodc}
In the theory $\zf$, the following two statements are equivalent:
\begin{enumerate}
\item $\dc$
\item $\fa_\omega(\mathcal P)$, where $\mathcal P$ is the class of 
all partial orders.
\end{enumerate}
\end{theorem}

The interested reader can find a full proof of Theorem~\ref{stevodc}, 
as well as of Theorem~\ref{stevoac} below, in~\cite[Theorem 3.2.4]{parente}. 
The next theorem that we mention characterizes the full Axiom of Choice, 
abbreviated $\ac$,  
as a collection of forcing axioms. Recall that, 
for a cardinal number $\kappa$, a partial order $\mathbb P$ is said 
to be \textbf{$\kappa$-closed} if every descending sequence 
$\langle p_\alpha\big|\alpha<\lambda\rangle$ of conditions 
of length $\lambda<\kappa$ has 
a lower bound\footnote{With this definition, every partial order 
is $\omega$-closed, but being $\kappa$-closed for some 
$\kappa\geq\omega_1$ is a nontrivial condition to impose 
on a partial order.}.

\begin{theorem}[Todor\v{c}evi\'c]\label{stevoac}
In the theory $\zf$, the following two statements are equivalent:
\begin{enumerate}
\item $\ac$
\item for every cardinal $\kappa$, $\fa_\kappa(\kappa\text{-closed})$.
\end{enumerate}
\end{theorem}

Other results in this line of research have been found 
by Gary Shannon~\cite{shannon}, who found characterizations 
of K\"onig's Lemma and of the principle 
of Countable Choice as forcing axioms.

The second section of this paper contains yet another characterization 
of $\ac$ as a forcing axiom. This characterization consists of explaining 
how Zorn's Lemma can be rephrased so that it looks like a legitimate 
forcing axiom. This technique can also be used to obtain another 
characterization of $\dc$ as well, and we also improve Shannon's 
characterization of K\"onig's Lemma~\cite[Theorem 2]{shannon}. We 
further provide a characterization of a further weak choice 
principle, which is in the spirit of~\cite[Corollary 2]{shannon}. 
Then in the third section, we prove what we consider to be the main 
result of this paper: we characterize the Boolean Prime Ideal theorem 
in terms of a statement that is very close to a forcing axiom. This 
statement allows to prove consequences of the Boolean Prime Ideal 
theorem by using the same type of reasonings that any forcing 
axiom allows, and show three examples of this.

\section{Axiom of Choice, K\"onig's Lemma, and linearly ordered sets}

In this section, we provide some characterizations of certain weak principles 
of choice, including the Axiom of Choice itself, as forcing axioms. We first 
introduce a definition that meshes together the partial orders that are 
typically used in forcing axioms, with those that concern Zorn's Lemma.

\begin{definition}
We will say that a partial order $\mathbb P$ is \textbf{semi-separative} if 
for every $p\in\mathbb P$, either $p$ is minimal or $p$ has two incompatible 
extensions in $\mathbb P$.
\end{definition}

Together with the previous definition, the following lemma will be very 
useful for our characterization of $\ac$.

\begin{lemma}\label{minimaltogeneric}
Let $\mathbb P$ be any semi-separative partial order, and $G\subseteq\mathbb P$. 
Then $G$ intersects every dense set in $\mathbb P$ if and only if 
$G=\{q\in\mathbb P\big|p\leq q\}$ for some minimal element $p$.
\end{lemma}

\begin{proof}
Suppose first that $G=\{q\in\mathbb P\big|p\leq q\}$ for some minimal element 
$p$, in particular $p\in G$. Note that, since $p$ is minimal, then $p\in D$ 
for every dense $D\subseteq\mathbb P$. Hence $G$ intersects every dense set.

Conversely, suppose that $G$ is a filter meeting every dense subset 
$D\subseteq\mathbb P$. Therefore the set $\mathbb P\setminus G$ cannot be 
dense, for it does not meet $G$. So there is an element $p\in G$ such that 
every extension of $p$ is an element of $G$. Thus, it cannot be the case 
that $p$ has two incompatible extensions, since any two elements of $G$ must 
be compatible. Since $\mathbb P$ is semi-separative, $p$ must then be minimal. 
We now claim that $G=\{q\in\mathbb P\big|p\leq q\}$. Since $p\in G$, clearly 
$\{q\in\mathbb P\big|p\leq q\}\subseteq G$, now to prove the converse 
inclusion, let $q\in G$. Since $G$ is a filter, there exists an $r\in G$ 
that extends both $p$ and $q$, so by minimality of $p$ we obtain that 
$r=p$ and so $p\leq q$ for every $q\in G$.
\end{proof}

\begin{definition}
We will say that a partial order $\mathbb P$ is a \textbf{Zorn partial order} 
if every linearly ordered subset of $\mathbb P$ has a lower bound.
\end{definition}

Lemma~\ref{minimaltogeneric} promptly allows us to prove the following 
characterization of the Axiom 
of Choice. Recall that Zorn's Lemma, which is equivalent to $\ac$, is the statement 
that every Zorn partial order has a minimal element (it is usually stated in terms 
of upper bounds and maximal elements, but of course both versions are equivalent).

\begin{theorem}\label{characterizationofac}
In the theory $\zf$, the following two statements are equivalent:
\begin{enumerate}
\item $\ac$,
\item $\fa_\infty(\mathcal Z)$, where $\mathcal Z$ is the class of 
all semi-separative Zorn partial orders.
\end{enumerate}
\end{theorem}

\begin{proof}\hfil
\begin{description}
\item[$1 \Rightarrow 2$] Since $\ac$ implies Zorn's Lemma, which 
asserts that every Zorn partial order has a minimal element, in 
particular every semi-separative Zorn partial order $\mathbb P$ has a minimal 
element. By Lemma~\ref{minimaltogeneric}, this minimal element 
gives rise 
to a fully generic (i.e. $\mathcal D$-generic where $\mathcal D$ is 
the collection of \textit{all} dense subsets of $\mathbb P$) filter 
$G\subseteq\mathbb P$.

\item[$2 \Rightarrow 1$] The fact that Zorn's Lemma implies $\ac$, 
together with Lemma~\ref{minimaltogeneric}, should finish the 
proof, but we will be a bit more explicit. So let $X$ be a 
family of nonempty sets. If $X_1=\{x\in X\big||x|=1\}$ and 
$X_2=X\setminus X_1$, 
and $f$ is any choice function on 
$X_2$, then $f\cup\{\langle x,\bigcup x\rangle\big|x\in X_1\}$ 
will be a choice function on $X$, hence we assume without 
loss of generality that every element of $X$ has at least 
two elements. Now let
\begin{equation*}
\mathbb P=\{f:Y\longrightarrow\bigcup X\big|Y\subseteq X\wedge f\text{ is a choice function on }Y\},
\end{equation*}
with the order given by $f\leq g$ iff $f\supseteq g$. It is easy to 
see that $\mathbb P$ is a semi-separative 
Zorn partial order (the fact that non-minimal elements have 
two incompatible extensions follows from our assumption that 
every $x\in X$ has at least two elements), so by hypothesis 
there exists a generic filter $G\subseteq P$, and clearly 
letting $h=\bigcup G$ will yield that 
$h:X\longrightarrow\bigcup X$ is a choice function on $X$ 
(the fact that $X=\dom(h)$ follows from the fact that 
$D_x=\{f\in\mathbb P\big|x\in\dom(f)\}$ is dense for every 
$x\in X$).
\end{description}
\end{proof}

We will omit the proof of the following theorem, since it 
is completely analogous to that of 
Theorem~\ref{characterizationofac}, once one remembers 
that $\dc$ is equivalent (under $\zf$) to the statement 
that every partial order $\mathbb P$ such that every 
linearly ordered subset $X\subseteq\mathbb P$ is finite, 
has a minimal element\footnote{This statement is referred 
to as ``Form 43L'' in~\cite[p. 31]{howard-rubin}.}.

\begin{theorem}
In the theory $\zf$, the following two statements are equivalent:
\begin{enumerate}
\item $\dc$,
\item $\fa_\infty(\mathcal F)$, where $\mathcal F$ is the class of 
all semi-separative partial orders such that every linearly ordered 
subset is finite.
\end{enumerate}
\end{theorem}

We now introduce some notation regarding weak choice principles. 
The symbol $\ac(\kappa,\lambda)$ denotes the statement that every family 
of cardinality at most $\kappa$, each of whose elements is nonempty 
and has cardinality at most $\lambda$, has a choice function. Variations 
of this notation, where instead of a cardinal we have something 
like $\mathrm{WO}$ which stands for well-orderable, should be given 
the obvious meaning. Our next theorem is, in fact, a small improvement 
over a theorem of Shannon~\cite[Theorem 2]{shannon}, who showed that K\"onig's Lemma 
is equivalent to a forcing axiom, for $\omega$ many dense sets, over 
a certain class of partial orders with a quite complex definition. 
The authors noticed that the description of the class of partial 
orders involved could be made much simpler by removing an unnecessary 
condition whose use in the corresponding proof can be skipped.

\begin{theorem}
In the theory $\zf$, the following four statements are equivalent:
\begin{enumerate}
\item $\ac(\omega,<\omega)$,
\item K\"onig's Lemma,
\item every countable union of finite sets is countable,
\item $\fa_\omega(\mathcal C)$, where $\mathcal C$ is the class of 
all partial orders whose underlying set is a countable union of 
finite sets.
\end{enumerate}
\end{theorem}

\begin{proof}\hfil
\begin{description}
\item[$1 \Leftrightarrow 2 \Leftrightarrow 3$] These equivalences are well-known (see 
e.g.~\cite[pp.19--20]{howard-rubin}).
\item[$3 \Rightarrow 4$]
If $\mathbb P\in\mathcal C$, then our assumption implies that 
$\mathbb P$ is countable, and in particular well-orderable. Hence 
we have a choice function on the powerset of $\mathbb P$, so whenever we have 
a sequence $\langle D_n\big|n<\omega\rangle$ of countably many dense sets, 
we can use the aforementioned choice function to recursively define a 
decreasing sequence $\langle p_n\big|n<\omega\rangle$ such that $p_n\in D_n$. 
Clearly $G=\{q\in\mathbb P\big|(\exists n<\omega)(p_n\leq q)\}$ will be a 
$\{D_n\big|n<\omega\}$-generic filter.

\item[$4 \Rightarrow 1$]
Let $\{X_n\big|n<\omega\}$ be a countable family of nonempty finite sets, 
indexed by $\omega$. We let 
\begin{equation*}
\mathbb P=\left\{f:n\longrightarrow\bigcup_{k<n}X_k\bigg|n<\omega\wedge f\text{ is a choice 
function on }\{X_k\big|k<n\}\right\}.
\end{equation*}
It can be easily verified that the partially ordered set 
$\mathbb P$ is the countable union 
of the finite sets $F_n=\{f\in\mathbb P\big|\dom(f)=n\}$. 
Hence if we consider, for every $n<\omega$, the dense set 
$D_n=\{p\in\mathbb P\big|n\in\dom(p)\}$, then we obtain a 
$\{D_n\big|n<\omega\}$-generic set $G$. It is straightforward to verify 
that $\bigcup G$ is a choice function on $\{X_n\big|n<\omega\}$.
\end{description}
\end{proof}

The last result of this section is very much in the spirit 
of~\cite[Corollary 2]{shannon}. This result consists of a 
characterization of the principle $\ac(\lo,<\omega)$, which 
asserts the existence of choice functions for every linearly 
orderable collection of nonempty finite sets. This weak 
choice principle is implied both by the Ordering Principle 
(asserting that every set can be linearly ordered), and 
by $\ac(\infty,<\omega)$ (which 
asserts the existence of a choice function on any 
arbitrary family of nonempty finite sets); and it 
implies $\ac(\wo,<\omega)$ (asserting the existence 
of a choice function on any well-orderable family 
of nonempty finite sets); and none of these 
implications is reversible~\cite[Corollary 4.6]{truss}.

In order to 
state the equivalence of this weak choice principle with 
something that resembles a forcing axiom, we will need to 
be more flexible with our notion of forcing axiom, and 
consider preorders instead of partial orders. Given 
$n<\omega$, we will denote by $\mathcal L_n$ the class 
of all preorders whose underlying set is the union of a 
pairwise disjoint linearly orderable family of finite 
sets, such that every 
antichain has size at most $n$. We will also denote 
$\mathcal L=\bigcup_{n<\omega}\mathcal L_n$. Also, 
a superscript $\lo$ in the statement of a forcing 
axiom will denote the additional assertion that the 
relevant filter can be taken to be linearly 
orderable.

\begin{theorem}\label{linearorder}
Under the theory $\zf$, the following four statements are equivalent:
\begin{enumerate}
\item $\ac(\lo,<\omega)$,
\item the union of a pairwise disjoint linearly orderable family of 
finite sets is linearly orderable,
\item $\fa_\infty^\lo(\mathcal L)$,
\item there exists an $n$, $1\leq n<\omega$, such that $\fa_\infty^\lo(\mathcal L_n)$ 
holds.
\end{enumerate}
\end{theorem}

\begin{proof}\hfil
\begin{description}
\item[$1\Rightarrow 2$] Let $X$ be a pairwise disjoint set, linearly 
ordered by $\leq$, each 
of whose elements is finite nonempty. Assuming $\ac(\lo,<\omega)$ we will 
proceed to construct a linear order $\preceq$ on $\bigcup X$. For each $x\in X$, 
define $L_x=\{L\subseteq x\times x\big|L\text{ is a linear order on }x\}$. 
Then each element in the linearly orderable family $L=\{L_x\big|x\in X\}$ 
is finite nonempty, so by $\ac(\lo,<\omega)$ it is possible to obtain 
a choice function $f$ on $L$. This allows us to define the relation
\begin{equation*}
y\preceq z \iff x_y\leq x_z\text{ or }(x_y=x_z\text{ and }y\ f(x_y)\ z)\text{ or }y=z,
\end{equation*}
(where $x_y$ denotes the unique $x\in X$ such that $y\in x$) on $\bigcup X$. 
It is straightforward to check that $\preceq$ linearly orders $\bigcup X$.

\item[$2 \Rightarrow 3$] Let $\mathbb P\in\mathcal L$. The first thing to 
notice, is that by assumption there exists a linear order $L$ on $\mathbb P$. 
Since $\mathbb P\in\mathcal L_n$ for some $n<\omega$, in particular the size 
of the antichains of $\mathbb P$ is bounded above. Thus we can take an antichain 
$A\subseteq\mathbb P$ of maximum possible cardinality. Take any $p\in A$, 
and define $G_p=\{q\big|q\not\perp p\}$. $G_p$ is linearly orderable by 
$L\upharpoonright G_p$, and we claim that $G_p$ is also a generic filter. 
If $D$ is any dense set, there exists a $q\in D$ with $q\leq p$, so 
$q\in G_p$ by definition and hence $G_p$ meets $D$. Now to see that $G_p$ 
is a filter, let $q,r\in G_p$. Then by definition both $q$ and $r$ are 
compatible with $p$, so there are extensions $q'\leq q$ and $r'\leq r$ that 
extend $p$ as well. If $q'$ and $r'$ were incompatible, the set 
$(A\cup\{q',r'\})\setminus\{p\}$ would be an antichain of cardinality strictly 
larger than $|A|$, contradicting that $A$ has maximum possible cardinality. 
Therefore $q'\not\perp r'$ and so $q\not\perp r$. Since 
$D=\{s\in\mathbb P\big|s\perp q\vee s\perp r\vee(s\leq q\wedge s\leq r)\}$ 
is a dense set, and we already showed that $G_p$ must intersect every dense 
set, and that any two elements in $G_p$ must be compatible, we can conclude 
that there exists an $s\in G_p$ such that $s\leq q$ and $s\leq r$. Thus 
$G_p$ is a generic linearly orderable filter.

\item[$3 \Rightarrow 4$] This is immediate.
\item[$4 \Rightarrow 1$] Let $X$ be a pairwise disjoint, linearly orderable 
family of nonempty finite sets. We preorder $\mathbb P=\bigcup X$ with the 
somewhat trivial preorder that makes $p\leq q$ for every $p,q\in\mathbb P$. 
Then every antichain in $\mathbb P$ is a singleton, so 
$\mathbb P\in\mathcal L_n$. Since every $x\in X$ is a dense set, when 
considered as a subset of $\mathbb P$, our hypothesis gives us a filter 
$G\subseteq\mathbb P$, meeting every $x\in X$, and equipped with a linear 
order $L$. Hence we can define 
$z_x=\min_L(G\cap x)$ for every $x\in X$ (it is possible to take $L$-minima 
because each $G\cap x$ is finite and nonempty), now the family 
$\{z_x\big|x\in X\}$ will be a selector for $X$.
\end{description}
\end{proof}

\section{A new characterization of the $\bpi$}

The Boolean Prime Ideal Theorem, denoted by $\bpi$, is well-known as a statement weaker 
than $\ac$ that still suffices to carry out many of the proofs that require $\ac$. In 
this section we will prove that the $\bpi$ is equivalent to a certain statement that 
deals with partially ordered sets and the existence of certain filter-like families 
therein. Afterwards we will show how this statement allows us to prove some 
consequences of $\bpi$ by using the same sort of ideas as in a proof that uses  
Zorn's Lemma, or Martin's Axiom. We will lay down a couple of definitions that 
we will need in order to state the version of the $\bpi$ that we will use for our 
proof.

\begin{definition}
Let $X$ be a nonempty set. A family $M$ consisting of functions from finite subsets 
of $X$ into $2$ will be called a \textbf{binary mess} on $X$ provided that it 
satisfies
\begin{enumerate}
\item $M$ is closed under restrictions, this is, if $s\in M$ and $F\subseteq\dom(s)$ 
 then $s\upharpoonright F\in M$, and
\item for every finite $F\subseteq X$, there exists an $s\in M$ such that 
 $\dom(s)=F$.
\end{enumerate}
\end{definition}

\begin{definition}
Let $X$ be a nonempty set, and let $M$ be a binary mess on $X$. We will say that 
a function $f:X\longrightarrow2$ is \textbf{consistent} with $M$ if for every 
finite $F\subseteq X$, $f\upharpoonright F\in M$.
\end{definition}

These definitions allow us to state the following weak choice principle, 
which is equivalent to the $\bpi$~\cite[Theorem 2.2]{jech}.

\begin{tabbing}
  The Consistency Principle \quad \= For every nonempty set $X$, and every 
binary mess $M$ \\
  \> on $X$, there exists a function $f:X\longrightarrow2$ that is \\
  \> consistent with $M$.\\
\end{tabbing}

In our search for an equivalent of $\bpi$, we were not able to find a 
statement which falls neatly into the template that we have been using 
for a forcing axiom. Inspired by Cowen's generalization of K\"onig's 
Lemma~\cite[Theorems 1 and 8]{cowen}, we were able to put together the 
equivalence in Theorem~\ref{maintheorem} below, though we first state 
a couple of definitions.

\begin{definition}
Let $\mathbb P$ be a partially ordered set.
\begin{enumerate}
\item We will say that a subset $G\subseteq\mathbb P$ is \textbf{2-linked} 
if every two elements of $G$ are compatible.
\item We will say that a family $\mathcal A$ of antichains of $\mathbb P$ 
is \textbf{centred} if for every choice of finitely many antichains 
$A_1,\ldots,A_n\in\mathcal A$, there exist elements $p\in\mathbb P$ and 
$a_1\in A_1,\ldots,a_n\in A_n$ such that $p\leq a_k$ for all 
$1\leq k\leq n$.
\end{enumerate}
\end{definition}

Under $\zfc$, every dense set gives rise to a maximal antichain, and 
viceversa. Additionally, the key property of filters is not so much 
that they are closed upwards, or that any two of its elements have a 
common extension in the filter itself, but rather, just that any two 
of its elements are compatible. This is what allows us to properly 
glue together the elements of a filter in order to obtain the desired 
object, in most of the proofs that use forcing axioms. Therefore 
we claim that the characterization in Theorem~\ref{maintheorem} 
below does not deviate excessively from the usual template of 
a forcing axiom.

\begin{theorem}\label{maintheorem}
The theory $\zf$ proves that the following two statements are 
equivalent:
\begin{enumerate}
\item $\bpi$,
\item for every partial order $\mathbb P$, and every centred family $\mathcal A$ 
of finite antichains, there exists an $\mathcal A$-generic 2-linked 
subset $G\subseteq\mathbb P$.
\end{enumerate}
\end{theorem}

\begin{proof}\hfil
\begin{description}
\item[$1 \Rightarrow 2$] Assume $\bpi$, which is equivalent to the 
Consistency Principle, and let $\mathbb P$ be a partial order, with a 
centred family $\mathcal A$ of finite antichains. Define a binary mess 
$M$ on $\mathbb P$ by
\begin{eqnarray*}
M & = & \left\{m:F\rightarrow2\big| F\in[\mathbb P]^{<\omega}
\wedge(\forall x,y\in F)(m(x)=m(y)=1\Rightarrow x\not\perp y)\right. \\
 & & \ \ \ \ \ \ \ \ \ \ \ \ \ \ \ \ \ \ \ \ \ \ \ \ \ \ \ \ \ \ \ \ \ \ \ \ \left. \wedge(\forall A\in\mathcal A)(A\subseteq F\Rightarrow(\exists x\in A)(m(x)=1))\right\}.
\end{eqnarray*}
It is readily checked that $M$ is closed under restrictions. Now if 
$F\subseteq\mathbb P$ is finite, then there are only finitely many 
elements $A\in\mathcal A$ with $A\subseteq F$, let those be 
$A_1,\ldots,A_n$. Since $\mathcal A$ is centred, there exists a $p\in\mathbb P$ 
and $a_k\in A_k$ for each $1\leq k\leq n$ such that $p\leq a_k$. We 
define $m:F\longrightarrow2$ by $m(x)=1\iff(\exists k\leq n)(x=a_k)$. 
It is readily checked that $m\in M$, thus $M$ is indeed a binary mess. 
So the Consistency Principle ensures the existence of a function 
$f:\mathbb P\longrightarrow2$ consistent with $M$. Letting 
$G=\{x\in\mathbb P\big|f(x)=1\}$ will give us an $\mathcal A$-generic 2-linked 
set.

\item[$2 \Rightarrow 1$] Let $M$ be a binary mess on some 
nonempty set $X$. We partially order $M$ itself by reverse 
inclusion, and for each finite $F\subseteq X$ we let 
$A_F=\{m\in M\big|\dom(m)=F\}$. Then each of the $A_F$ is a 
finite antichain (in fact, $|A_F|\leq 2^{|F|}$). We define 
$\mathcal A=\{A_F\big|F\in[X]^{<\omega}\}$, and proceed to 
verify that $\mathcal A$ is indeed a centred family of 
antichains, so let $A_{F_1},\ldots,A_{F_n}\in\mathcal A$. 
Since $M$ is a binary mess, there exists an $m\in M$ with 
$\dom(m)=F_1\cup\cdots\cup F_n$. For each $1\leq k\leq n$, 
we have that $m_k=m\upharpoonright F_k\in A_{F_k}$ is an 
element extended by $m$, thus $\mathcal A$ is a centred 
family. Therefore, by assumption there exists an 
$\mathcal A$-generic 2-linked family $G$.

Now we claim that defining $f=\bigcup G$ yields a function 
consistent with $M$. Note first that $f$ must be a function, 
since $G$ is linked. Furthermore, for each $x\in X$, $G$ must 
meet $A_{\{x\}}$, which implies that $x\in\dom(f)$. Hence 
$\dom(f)=X$. Lastly, $f$ is consistent with $M$ because, for 
each finite $F\subseteq X$, $G$ must intersect $A_F$, and it 
is easy to see that $G\cap A_F$ consists of the single 
element $f\upharpoonright F\in M$.
\end{description}
\end{proof}

We now present three examples of proofs using the equivalence 
found in Theorem~\ref{maintheorem}, in order to illustrate 
how we can use this new equivalence to prove consequences 
of $\bpi$ in the spirit of proofs that use forcing axioms.
Our first example is the proof of the Ordering Principle.

\begin{example}
The ordering principle is the statement that every set can 
be linearly ordered. This principle is implied by the $\bpi$, 
and the implication is not reversible\footnote{Mathias~\cite{mathias} 
proved that the Ordering Principle does not imply the Order Extension 
Principle (the statement that every partial ordering can be extended 
to a total ordering), which is another consequence of $\bpi$.}. 
Thus we will prove the ordering principle, in $\zf$, under the 
assumption that statement (2) in Theorem~\ref{maintheorem} 
holds.

So let X be an arbitrary (nonempty) set. Partially order 
the set
\begin{equation*}
\mathbb P=\{L\big|L\text{ is a linear order, and }\dom(L)\in [X]^{<\omega}\}
\end{equation*}
by reverse inclusion, so that $L\leq L'$ iff $L\supseteq L'$. 
For each $F\in [X]^{<\omega}$ we let $A_F$ be the collection 
of all linear orders on $F$, which is a finite (in fact, of 
size $|F|!$) antichain in $\mathbb P$. So the family 
$\mathcal A=\{A_F\big|F\in[X]^{<\omega}\}$ consists of 
finite antichains; we will now proceed to prove that it is 
centred, so consider finitely many elements 
$A_{F_1},\ldots,A_{F_n}\in\mathcal A$. Let $L$ be a 
linear order on the finite set $F_1\cup\cdots\cup F_n$, 
then $L$ simultaneously extends each of the elements 
$L\upharpoonright F_k\in A_{F_k}$, for $1\leq k\leq n$, 
and therefore $\mathcal A$ is linked. Thus we obtain 
an $\mathcal A$-generic 2-linked set $G$. It is readily 
checked that $\bigcup G$ is, in fact, a linear order 
on $X$.
\end{example}

\begin{example}
We shall now consider a statement which is actually known to be 
equivalent to the $\bpi$. We will show only half of the equivalence, 
the half that illustrates the use of our new characterization of 
the $\bpi$. 

Let $\{A_i\big|i\in I\}$ be a collection of finite sets, and let $S$ 
be a symmetric binary relation on $A=\bigcup_{i\in I} A_i$. We will 
say that a function $f$ with range contained in $A$ is \textbf{$S$-consistent} 
if $(\forall x,y\in\dom(f))(f(x)\ S\ f(y))$. The statement that we will 
prove, assuming $\bpi$, is the following: If for every finite $W \subset I$ 
there is an $S$-consistent choice function for $\{A_i\big|i \in W\}$, then 
there is an $S$-consistent choice function for the whole family 
$\{A_i\big|i \in I\}$.\footnote{The fact that the $\bpi$ is equivalent 
to this statement for all collections of finite sets $\{A_i\big|i\in I\}$ 
and all symmetric relations $S$ on $\bigcup_{i\in I}A_i$ is proved 
in~\cite[Theorem 2*]{los-ryll}.}

For this, we let 
\begin{eqnarray*}
\mathbb P & = & \{p\big|p:W\longrightarrow A\text{ for some finite }W\subseteq I\text{ and }\\
 & & \ \ \ \ \ \ p\text{ is an }S\text{-consistent choice function on }\{A_i\big|i\in W\}\},
\end{eqnarray*}
and we partially order $\mathbb P$ by reverse extension (i.e. 
$p\leq q$ iff $p\supseteq q$). For each finite $W\subseteq I$, 
the family $A_W=\{p\in\mathbb P\big|\dom(p)=W\}$ is clearly a 
finite antichain, and for each finite collection of these 
antichains, $A_{W_1},\ldots,A_{W_n}$, it is clear that any 
$S$-consistent choice function $p$ on $\{A_i\big|i\in W_1\cup\cdots\cup W_n\}$ 
(there exists at least one by hypothesis) will extend the 
elements $p\upharpoonright W_i\in A_{W_i}$ for every $1\leq i\leq n$. 
Therefore, letting $\mathcal A=\{A_W\big|W\in[I]^{<\omega}\}$ yields 
a centred family of finite antichains. Thus by the assumption that 
$\bpi$ holds and Theorem~\ref{maintheorem}, we can obtain an 
$\mathcal A$-generic 2-linked family $G$. We claim that $f=\bigcup G$ 
is an $S$-consistent choice function on $\{A_i\big|i\in I\}$. It is 
easy to see that $f$ is a function because $G$ is 2-linked. Moreover, 
$\dom(f)=I$ since $G$ must intersect each of the $A_{\{i\}}$, for all 
$i\in I$. Finally, given any two $i,j\in I$, we can derive from the 
fact that $G$ meets $A_{\{i,j\}}$ that 
$f\upharpoonright\{i,j\}\in G\subseteq\mathbb P$, and therefore we 
must have that $f(i)\ S\ f(j)$. Hence $f$ is $S$-consistent, and we are 
done.
\end{example}

Our last example will be a proof of the Hahn-Banach theorem, 
which is an extremely well-known result. This theorem is 
implied by the $\bpi$~\cite[Section 5]{los-ryll}, though the implication 
is not reversible~\cite{pincus}.

\begin{example}
We will prove the Hahn-Banach theorem from statement (2) in 
Theorem~\ref{maintheorem}, so we will need to introduce some 
terminology. If $V$ is a real vector space, a \textbf{Minkowski 
functional} 
on $V$ is a function $p:V\longrightarrow\mathbb R$ satisfying 
$p(x+y)\leq p(x)+p(y)$ and $p(tx)=tp(x)$ for every $x,y\in V$ 
and every positive $t\in\mathbb R$. A \textbf{linear functional} 
is simply a linear transformation $f:V\longrightarrow\mathbb R$. 
The Hahn-Banach theorem is the following statement: for every 
real vector space $V$, for every Minkowski functional 
$p:V\longrightarrow\mathbb R$ 
on $V$ and every linear functional $f:W\longrightarrow\mathbb R$ 
defined on a subspace $W$ of $V$, such that $(\forall x\in W)(f(x)\leq p(x))$, 
there exists a linear functional $\hat{f}:V\longrightarrow\mathbb R$ 
extending $f$ such that $(\forall x\in V)(\hat{f}(x)\leq p(x))$.

In order to carry out our proof, we would like to define a partial 
order of approximations to the desired functional, but doing this 
carelessly gets us at risk of not being able to 
obtain small enough antichains. So our approximations will 
consist on \textit{approximating} (though not specifying) the 
value of the desired functional on finitely many vectors, for 
which we need to deal with certain special kinds of intervals. 
Given a closed interval $I=[a,b]\subseteq\mathbb R$ and 
$n<\omega$, we will define a \textbf{dyadic 
subinterval of depth $n$} of $I$ to be any 
interval of the form $\left[a+\frac{(b-a)k}{2^n},a+\frac{(b-a)(k+1)}{2^n}\right]$ 
where $k\in\{0,1,\ldots,2^n-1\}$. A \textbf{dyadic subinterval} 
of $I$ will be a dyadic subinterval of depth some $n<\omega$; 
we denote by $I_n$ the collection of all ($2^n$ many) dyadic 
subintervals of depth $n$ of $I$, and by $I_\infty=\bigcup_{n<\omega}I_n$ 
the collection of all dyadic subintervals of $I$ (of any depth). 
Dealing with dyadic subintervals will be very useful because 
there are only finitely many dyadic subintervals of any given 
depth, and because, given any closed interval $I$, any two of 
its dyadic subintervals $J,J'\in I_\infty$ will satisfy that 
either $|J\cap J'|\leq 1$, or $J\subseteq J'$ or $J'\subseteq J$.

We can now finally proceed to the proof. Assume the 
corresponding hypotheses, i.e. let $V$ be a linear 
vector space, $p:V\longrightarrow\mathbb R$ a 
Minkowski functional, and $f:W\longrightarrow\mathbb R$ 
be a linear functional, defined on the (proper) subspace $W$ 
of $V$, such that $(\forall x\in W)(f(x)\leq p(x))$. For 
each $x\in V$ we define the closed interval 
$I(x)=[-p(-x),p(x)]$. Note 
that if we manage to define $\hat{f}$ in such a way that 
$\hat{f}(x)\in I(x)$ for every $x$, we will have taken 
care of the inequality requirement for $\hat{f}$ (and 
hence we would only need to work towards ensuring that 
$\hat{f}$ is linear). Thus we define our partially ordered 
set $\mathbb P$ to consist of all functions $g$ with 
domain some finite $X\subseteq V$ that satisfy
\begin{enumerate}
\item $(\forall x\in X)(g(x)\in I(x)_\infty)$,
\item $(\forall x\in X)(x\in W\Rightarrow f(x)\in g(x))$, 
\item $(\forall x,y\in X)(x+y\in X\Rightarrow (g(x)+g(y))\cap g(x+y)\neq\varnothing)$, and 
\item $(\forall x\in X)(\forall r\in\mathbb R)(rx\in X\Rightarrow rg(x)\cap g(rx)\neq\varnothing$.
\end{enumerate}
Hence every $g\in\mathbb P$ specifies, for a finite number 
of $x\in V$, a dyadic subinterval of $I(x)$ from which we 
intend to eventually pick the value $\hat{f}(x)$, in 
a way that is consistent with the fact that we want $\hat{f}$ 
to be a linear functional extending $f$. The partial ordering 
is given by: $g\leq g'$ iff $\dom(g')\subseteq\dom(g)$ and 
$(\forall x\in \dom(g'))(g(x)\subseteq g'(x))$. For each 
function $h:X\longrightarrow\omega$, with domain some 
$X\in[V]^{<\omega}$, 
we let 
$A_h=\{g\in\mathbb P\big|\dom(g)=X\text{ and }(\forall x\in X)(g(x)\in I(x)_{h(x)})\}$.

\begin{claim}\label{nonemptyness}
Each of the $A_h$ is a nonempty finite antichain.
\end{claim}

\begin{proof}[Proof of Claim~\ref{nonemptyness}]
The fact that $A_h$ is a finite antichain follows directly 
from the fact that the interval $I(x)$ has only finitely 
many dyadic subintervals of depth $h(x)$, for each of the 
finitely many $x\in\dom(h)$ (and that the intersection of 
any two dyadic intervals of the same depth is either empty 
or a singleton). The nontrivial part of the claim is thus 
the nonemptyness of $A_h$. For this, we will use the following 
fact (which we will not prove because it properly belongs 
to Functional Analysis rather than Set Theory): for every 
linear functional $l:W'\longrightarrow\mathbb R$ defined 
on some subspace $W'$ of $V$, satisfying 
$(\forall x\in W')(l(x)\leq p(x))$, and for every 
$z\in V\setminus W'$, it is possible to extend $l$ to a 
linear functional $l':W'+\mathbb R z\longrightarrow\mathbb R$ such that 
for every $x\in W'+\mathbb R z$, 
$l'(x)\leq p(x)$\footnote{A proof of this fact can be found 
in~\cite[2.3.2, p. 57]{pedersen}.}. Thus, 
proceeding by induction, it can be shown that for every 
finite $X\subseteq V$ there is an extension $f'$ of $f$, 
defined in the subspace $W'$ generated by $W\cup X$, such 
that $(\forall x\in W')(f'(x)\leq p(x))$. Consequently, 
by linearity of $f'$ and positive homogeneity of $p$, 
$f'(x)\in[-p(-x),p(x)]=I(x)$, so we can let $g(x)$ be 
the leftmost dyadic subinterval of depth $h(x)$ of $I(x)$ 
containing $f'(x)$, and this way we have defined an 
element $g\in A_h$.
\end{proof}

We now define
\begin{equation*}
\mathcal A=\{A_h\big|h:X\longrightarrow\omega\text{ for some }X\in[V]^{<\omega}\},
\end{equation*}
and note that this is a centred family of finite antichains. 
For if we 
are given finitely many functions 
$h_1:X_1\longrightarrow\omega,\ldots,h_n:X_n\longrightarrow\omega$, 
we can define $h:(X_1\cup\cdots\cup X_n)\longrightarrow\omega$ 
by $h(x)=\max\{h_i(x)\big|1\leq i\leq n\text{ and }x\in X_i\}$, 
and take a $g\in A_h$ by Claim~\ref{nonemptyness}. Now for each 
$1\leq i\leq n$, and each 
$x\in X_i$, we pick a dyadic interval $g_i(x)$ of depth $h_i(x)$ 
containing the interval $g(x)$ (recall that $g(x)$ is a dyadic interval 
of depth $h(x)\geq h_i(x)$, so there is a unique such interval). 
Then it is clear that $g$ 
extends the element $g_i\in A_{h_i}$.

Hence we can invoke an $\mathcal A$-generic 2-linked set $G$. 
We define $\hat{f}:V\longrightarrow\mathbb R$ as follows: for 
every $x\in V$, we let 
$\mathcal I_x=\{g(x)\big|g\in G\text{ and }x\in\dom(g)\}\subseteq I_\infty$. 
For each $n<\omega$, since $G$ must intersect $A_{\{\langle x,n\rangle\}}$, 
it follows that $\mathcal I_x$ contains at least one dyadic 
subinterval of $I(x)$ of depth $n$; and since $G$ is 2-linked this 
interval is, in fact, unique. Now if $n<m$ and $J,J'\in\mathcal I_x$ 
are the two dyadic subintervals of depths $n$ and $m$, respectively, 
then $J'\subseteq J$. Hence the family $\mathcal I_x$ 
can be thought of as a nested sequence of closed intervals, with 
arbitrarily small diameters; therefore there is a 
unique real number $\hat{f}(x)\in\bigcap\mathcal I_x$. By construction, 
$\hat{f}(x)\in I(x)$, so $\hat{f}(x)\leq p(x)$ for every $x\in V$. 
Also, given $x\in W$, we have that $f(x)\in\bigcap\mathcal I_x$, 
thus $\hat{f}(x)=f(x)$, so the function $\hat{f}$ actually 
extends $f$. We 
now proceed to show that $\hat{f}$ is a linear functional, so let 
$x,y\in V$ and suppose towards a contradiction that 
$\hat{f}(x)+\hat{f}(y)\neq\hat{f}(x+y)$. Then we can pick an $n<\omega$ 
so large, that no interval of length $\frac{p(x+y)+p(-x-y)}{2^n}$ 
containing $\hat{f}(x+y)$ can intersect an interval of length 
$\frac{p(x)+p(-x)+p(y)+y(-y)}{2^{n-1}}$ containing 
$\hat{f}(x)+\hat{f}(y)$. But then, letting $h$ be the function 
constantly $n$ with domain $\{x,y,x+y\}$, there must be a 
$g\in G\cap A_h$; and so on the one hand we must have 
$g(x)\in I(x)_n,g(y)\in I(y)_n,g(x+y)\in I(x+y)_n$ (so that 
$g(x)$, $g(y)$, and $g(x+y)$ are intervals of lengths 
$\frac{p(x)+p(-x)}{2^n}$, $\frac{p(y)-p(-y)}{2^n}$, 
and $\frac{p(x+y)+p(-x-y)}{2^n}$, respectively), and 
$(g(x)+g(y))\cap g(x+y)\neq\varnothing$; while 
symultaneously $\hat{f}(x)+\hat{f}(y)\in g(x)+g(y)$ 
and $\hat{f}(x+y)\in g(x+y)$, which is a contradiction. 
Hence $\hat{f}(x)+\hat{f}(y)=\hat{f}(x+y)$; and in a 
completely analogous way the reader can verify that 
$\hat{f}(rx)=r\hat{f}(x)$ for every $x\in V$ and 
every $r\in\mathbb R$. This finishes the proof.
\end{example}

\section*{Acknowledgments}
The first author was partially supported by Postdoctoral Fellowship 
number 263820 from the Consejo Nacional de Ciencia y Tecnolog\'{\i}a 
(CONACyT), Mexico. The second author was supported by grant number 
DMS-1307164 of the National Science Foundation, as part of the summer 
Research Experience for Undergraduates program at the Department of 
Mathematics, University of Michigan.


\begin{thebibliography}{99}

\bibitem{abraham}
Abraham, U.
\newblock ``Proper Forcing'',
\newblock in {\em Handbook of Set Theory}, M. Foreman and A. Kanamori (eds.), 
Springer, 2010; pp. 333--394.

\bibitem{bell}
Bell, M.
\newblock ``On the combinatorial principle $P(c)$'',
\newblock {\em Fundamenta Mathematicae} \textbf{114} (1981), 
pp. 149--157.

\bibitem{cowen}
Cowen, R.,
\newblock ``Generalizing K\"onig's Infinity Lemma'',
\newblock {\em Notre Dame Journal of Formal Logic}
\textbf{18} (1977),
  pp.~243--247.
  
\bibitem{goldblatt}
Goldblatt, R.,
\newblock ``On the role of the Baire category theorem and Dependent 
Choice in the foundations of logic'',
\newblock {\em Journal of Symbolic Logic} 
\textbf{50} no.~2 (1985), 412--422.

\bibitem{howard-rubin}
Howard, P. and Rubin, J.,
\newblock {\em Consequences of the Axiom of Choice},
\newblock American Mathematical Society, 1991.

\bibitem{jech}
Jech, T.,
\newblock {\em The Axiom of Choice},
\newblock North Holland Publishing Company, Amsterdam-London, 1973.

\bibitem{kunen}
Kunen, K.,
\newblock {\em Set Theory},
\newblock North Holland Publishing Company, 1983.

\bibitem{los-ryll}
\L o\'s, J. and Ryll-Nardzewski, C.,
\newblock ``On the application of Tychonoff's theorem in mathematical proofs'',
\newblock {\em Fundamenta Mathematicae} \textbf{38} (1951), 233--237.

\bibitem{maryanthe1}
Malliaris, M. and Shelah, S., 
\newblock ``Cofinality  spectrum  theorems  in model theory, 
set theory, and general topology'',
\newblock  {\em Journal of the American Mathematical Society} \textbf{29} 
(2016), 237--297.

\bibitem{maryanthe2}
Malliaris, M. and Shelah, S., 
\newblock ``General topology meets model theory, on 
$\mathfrak p$ and $\mathfrak t$'',
\newblock {\em Proceedings of the National Academy of Sciences USA}
\textbf{110} (2013), 13300--13305.

\bibitem{mathias}
Mathias, A.,
\newblock ``The order extension principle'',
\newblock in {\em Proceedings of Symposia in Pure Mathematics vol. 13, 
part II: Axiomatic Set Theory}, T. Jech (ed.), American Mathematical 
Society, 1974.

\bibitem{parente}
Parente, F.,
\newblock ``Boolean valued models, saturation, forcing axioms'',
\newblock Master's Thesis, Universit\`a di Pisa, 2015.

\bibitem{pedersen}
Pedersen, G.,
\newblock {\em Analysis Now},
\newblock Graduate Texts in Mathematics 118, Springer, 1989.

\bibitem{pincus}
Pincus, D.,
\newblock ``Independence of the prime ideal theorem from the Hahn Banach theorem'',
\newblock {\em Bulletin of the American Mathematical Society} 
\textbf{78} no.~5 (1972), 766--770.

\bibitem{shannon}
Shannon, G.
\newblock ``Provable Forms of Martin's Axiom'',
\newblock {\em Notre Dame Journal of Formal Logic} \textbf{31} (1990),
  pp.~382--388.
\bibitem{shelah}
Shelah, S.,
\newblock ``Proper and Improper Forcing'',
\newblock Perspectives in Mathematical Logic, Springer, 1998.

\bibitem{truss}
Truss, J.,
\newblock ``Finite axioms of choice'',
\newblock {\em Annals of Mathematical Logic} \textbf{6} (1973), 147--176.

\end{thebibliography}
\end{document}